\documentclass[11pt, reqno]{amsart}
\usepackage{amsmath,amssymb,amsfonts,amsthm,graphicx,mathrsfs,mathtools}
\usepackage[margin=1.5in]{geometry}
\usepackage[dvipsnames]{xcolor}
\usepackage{graphicx,tipa}
\usepackage{tikz}
\usepackage{graphicx,amsmath,calc}
\usepackage{subcaption}
\makeatletter
\@namedef{subjclassname@2020}{%
\textup{2020} Mathematics Subject Classification}
\makeatother

\title[minimal surfaces in the doubled Schwarzschild $3$-manifold]{complete minimal surfaces of finite topology in the doubled Schwarzschild $3$-manifold}
\author[J. Choe]{Jaigyoung Choe}
\address[]{Jaigyoung Choe, Korea Institute for Advanced Study, 85 Hoegiro, Dongdaemun-gu, Seoul 02455, Republic of Korea}
\email{choe@kias.re.kr}
\author[J. Lee]{Jaehoon Lee}
\address[]{Jaehoon Lee, School of Mathematics, Korea Institute for Advanced Study, 85 Hoegiro, Dongdaemun-gu, Seoul 02455, Republic of Korea}
\email{jaehoonlee@kias.re.kr}

\author[E. Yeon]{Eungbeom Yeon}
\address[]{Eungbeom Yeon, Department of Mathematical Sciences, Pusan National University, Busan 46241, Republic of Korea}
\email{ebeom.yeon@pusan.ac.kr}

\begin{document}

\newtheorem{theorem}{theorem}[section]
\newtheorem{thm}[theorem]{Theorem}
\newtheorem{lemma}[theorem]{Lemma}
\newtheorem{cor}[theorem]{Corollary}
\newtheorem{prop}[theorem]{Proposition}
\newtheorem{rmk}[theorem]{Remark}
\newtheorem{Question}[theorem]{Question}
\newtheorem{conj}[theorem]{Conjecture}
\newtheorem*{mainthm1}{Main Theorem}
\newtheorem*{mainthm2}{Theorem 2}

\renewcommand{\theequation}{\thesection.\arabic{equation}}
\newcommand{\RNum}[1]{\uppercase\expandafter{\romannumeral #1\relax}}
\newcommand{\R}{\mathbb{R}}
\newcommand{\C}{\mathbb{C}}
\newcommand{\grad}{\nabla}
\newcommand{\laplacian}{\Delta}
\newcommand{\tgamma}{\tilde{\gamma}}
\newcommand{\ttau}{\tilde{\tau}}
\newcommand{\pp}{\Phi}
\newcommand{\tkappa}{\tilde{\kappa}}
\renewcommand{\d}{\textup{d}}
\newlength{\mywidth}
\newcommand\bigfrown[2][\textstyle]{\ensuremath{%
  \array[b]{c}\text{\resizebox{\mywidth}{.7ex}{$#1\frown$}}\\[-1.3ex]#1#2\endarray}}
\newcommand{\arc}[1]{{%
  \setbox9=\hbox{#1}%
  \ooalign{\resizebox{\wd9}{\height}{\texttoptiebar{\phantom{A}}}\cr#1}}}


\begin{abstract}
We construct a complete embedded minimal surface with arbitrary genus in the doubled Schwarzschild $3$-manifold. A
classical desingularization method is used for the construction.
\end{abstract}

\maketitle

\section{Introduction}\label{intro}
\setcounter{equation}{0}
The Schwarzschild manifold is a solution to the Einstein equation, representing an important example of asymptotically flat manifolds. Throughout history, intensive research on minimal submanifolds in these manifolds has yielded significant results in geometry, including the positive mass theorem. Hence, investigating various types of minimal submanifolds in asymptotically flat manifolds is of great importance. However, only a limited number of examples are currently known, and in particular, no complete minimal surfaces with positive genus have been discovered in the (doubled) Schwarzschild $3$-manifold.

In this paper, we present the construction of new complete embedded minimal surfaces in the doubled Schwarzschild $3$-manifold $\widehat{M}$ with arbitrary genus:
\begin{mainthm1}[Theorem \ref{mainthm}]
For each integer $\tau\geq1$, there exists a complete embedded minimal surface $\Sigma_{\tau}\subset\widehat{M}$ of genus $\tau$. $\Sigma_\tau$ has finite total curvature and quadratic area growth, and is asymptotic to a totally geodesic plane.
\end{mainthm1}
The key idea in our construction is to desingularize the union of the horizon and the totally geodesic plane along the intersection curve using an appropriate reflection method in $\widehat{M}$. This desingularization technique has been utilized  in various studies (\cite{Choe, HM2, HMW,L}). It is very interesting to compare our existence result with some non-existence results in the Schwarzschild $3$-manifold recorded in the literature (\cite{C, CCE, CM, SY}).

This paper is organized as follows. Section \ref{pre} provides historical remarks on the subject, along with fundamental facts about the doubled Schwarzschild $3$-manifold. In Section \ref{sec3}, we discuss the geometric details of the area-minimizing disk, offering insights into the solution of Plateau's problem for the given contour. Section \ref{sec4} focuses on the curvature estimate of the area-minimizing disk, while the focus shifts to the limit behavior in Section \ref{sec5}. Finally, in Section \ref{sec6}, we construct the desired surfaces by patching together isometric copies of the limit surface.

\section*{Acknowledgements}
JC supported in part by RS-2023-00246133; JL supported in part by a KIAS Individual Grant MG086401 at Korea Institute for Advanced Study; EY supported in part by National Research Foundation of Korea NRF-2022R1C1C2013384 and NRF-2021R1A4A1032418.


\section{Preliminaries}\label{pre}
\setcounter{equation}{0}
\subsection{Historical remarks}
The Riemannian Schwarzschild manifold has been extensively studied throughout history as it represents a solution to the Einstein field equation. Mathematically, the Riemannian Schwarzschild manifold $(M^n,g^n_{sch})$ is defined as follows:
\begin{align*}
M^n=\left\{ x \in \mathbb{R}^n : |x| \ge \left( \frac{m}{2}\right)^{\frac{1}{n-2}} \right\},\ g^n_{Sch}=\left( 1+\frac{m}{2|x|^{n-2}}\right)^{\frac{4}{n-2}}g_{\mathbb{R}^n}.
\end{align*}
This manifold exhibits spherical symmetry and shows an asymptotically flat behavior at infinity. The boundary of the manifold, known as the horizon, is a totally geodesic hypersurface in $M^n$.

The significance of minimal submanifolds in asymptotically flat manifolds was initially recognized in the pioneering work of Schoen and Yau \cite{SY}. They proved that asymptotically planar stable minimal hypersurfaces cannot exist in asymptotically Schwarzschildean manifolds. This result played a crucial role in establishing the well-known positive mass theorem, which has greatly influenced subsequent research. In \cite{C}, the author further generalized the result to show the non-existence of complete stable properly embedded minimal hypersurfaces in asymptotically flat manifolds. Naturally, these non-existence results raise the question of what kinds of minimal submanifolds would exist in $M^n$. As a matter of fact, examples of minimal submanifolds in $M^n$ are very rare in general.

Considering a $2$-dimensional complete embedded minimal end in $M^n$ with finite total curvature and quadratic area growth, the result of \cite{BR} implies that it must either be bounded or has a logarithmic growth. This leads to the intriguing question of whether complete minimal surfaces with such ends can exist. In \cite{CM}, non-existence results for minimal surfaces in asymptotically Schwarzschildean $3$-manifolds were given in this regard. The authors demonstrated that no minimal surfaces, perturbing the Euclidean catenoid, exists in asymptotically Schwarzschildean $3$-manifolds. Furthermore, \cite{CCE} established that any slab bounded by complete minimal surfaces in an asymptotically flat $3$-manifold must be a Euclidean slab if the manifold possesses the horizon and non-negative scalar curvature. These results consistently show that the presence of the horizon poses a substantial obstacle to the existence of complete minimal examples. However, in the absence of the horizon, complete minimal planes were obtained in \cite{CK}.

Given that the horizon obstructs the existence of minimal surfaces, it is worthwhile to consider the doubled Schwarzschild $3$-manifold without a boundary, which can be obtained by the inversion with respect to the horizon. In \cite{B}, Brendle showed the existence of a minimal sphere near the horizon. Additionally, as the inversion across the horizon is an isometry, examples of free boundary minimal surfaces with their boundaries supported on the horizon are closely related. Nonetheless, to date, no information has been provided regarding the existence of minimal surfaces with arbitrary genus.
\subsection{Doubled Schwarzschild manifolds}
Consider the Riemannian manifold $\widehat{M}=\mathbb{R}^3\backslash\{(0,0,0)\}$ equipped with the metric $\widehat{g}:=\left(1+\frac{m}{2|x|}\right)^{4}g_{\mathbb{R}^3}$, where $g_{\mathbb{R}^3}$ denotes the Euclidean metric and $|\cdot|$ represents the Euclidean norm. Let us define a map $I: \widehat{M}\to\widehat{M}$ given by
\begin{align*}
I(x)=\left(\frac{m}{2}\right)^{2}\cdot\frac{x}{|x|^2}, \ \forall x\in\mathbb{R}^3\backslash\{(0,0,0)\},
\end{align*}
which is an inversion with respect to the Euclidean sphere $\left\{x\in\mathbb{R} ^3:|x|=\frac{m}{2}\right\}$. It can be shown that $I$ induces an isometry of $\widehat{M}$ by exploiting its conformal properties in Euclidean space. Specifically, $I$ is conformal in Euclidean space with the conformal factor $\left(\frac{m}{2}\right)^{4}\frac{1}{|x|^4}$ and satisfies the relation
\begin{align*}
\left(1+\frac{m}{2\left|I(x)\right|} \right)^{4}\cdot\left(\frac{m}{2}\right)^{4}\frac{1}{|x|^4}=\left(1+\frac{m}{2|x|}\right)^{4}.
\end{align*}
Furthermore, the Riemannian Schwarzschild $3$-manifold $(M, g)$ can be defined as
\begin{align*}
M=\left\{x\in\mathbb{R}^3:|x|\geq \tfrac{m}{2}\right\},\ g=\left(1+\tfrac{m}{2|x|}\right)^{4}g_{\mathbb{R}^3},
\end{align*}
and we can express as the union $\widehat{M}=M\cup I(M)$. Thus, it is appropriate to refer to $\widehat{M}$ as the \emph{doubled Schwarzschild $3$-manifold}.

Within $M$, there exist certain totally geodesic surfaces as fixed point sets under isometries, namely:
\begin{itemize}
\item[-] The horizon $\partial M^3=\left\{x\in\mathbb{R}^n:|x|=\frac{m}{2}\right\}$.
\item[-] Planes passing through the origin.
\end{itemize}
It can be proven that these are the only totally geodesic surfaces in $\widehat{M}$. In fact, a surface being totally geodesic in $\widehat{M}$ implies it is totally umbilic in Euclidean space, thus belonging to either a sphere or a plane. Among spheres and planes, it can be observed that only the horizon and planes passing through the origin are totally geodesic in the doubled Schwarzschild $3$-manifold. As the metric on $\widehat{M}$ possesses radial symmetry, reflections along totally geodesic planes can be regarded as isometries in $\widehat{M}$ and will be extensively utilized throughout this paper. Although minor modifications in the conformal factor would lead to the doubled Schwarzschild $n$-manifold, we confine our discussion to the $3$-dimensional case.

\section{Geometric properties of the area-minimizing disk}\label{sec3}
\setcounter{equation}{0}
We consider a particular contour $\Gamma_{\theta, R}$ ($\theta\in(0,\frac{\pi}{2}]$, $R>\frac{m}{2}$) in the Schwarzschild $3$-manifold $M$, which plays a fundamental role in this paper. For each $\alpha\in[0,2\pi)$, we denote $Q_\alpha$ as a totally geodesic plane passing through the origin and perpendicular to the vector $(-\sin\alpha, \cos\alpha, 0)$. Likewise, for each $\phi \in [0,\frac{\pi}{2})$, we denote $P_{\alpha,\phi}$ as a totally geodesic plane passing through the origin and perpendicular to $(-\sin\phi\cos\frac{\alpha}{2}, -\sin\phi\sin\frac{\alpha}{2},\cos\phi)$. Since all planes $P_{\alpha,0}$ are identical, we can simply refer to it as $P_0$. The contour $\Gamma_{\theta, R}$ comprises three distinct arcs: $\gamma_{0, R}, \gamma_{\theta, R}$, and $C_{\theta,R}$. Here, $\gamma_{\alpha, R}$ is a union of two geodesic arcs connecting $(0,0,\frac{m}{2})$, $(\frac{m}{2}\cos\alpha, \frac{m}{2}\sin\alpha, 0)$, and $(R\cos\alpha, R\sin\alpha, 0)$. Furthermore, $C_{\theta,R}$ is a circular arc connecting the endpoints of $\gamma_{0,R}$ and $\gamma_{\theta, R}$. It is worth noting that $\gamma_{\alpha, R}$ lies in the plane $Q_{\alpha}$. Figure 1 illustrates the contour $\Gamma_{\theta, R}$.
\begin{figure}
\centering
\begin{tikzpicture}
\shade[ball color = gray!40, opacity = 0.4] (0,0) circle (2cm);
\draw (-2,0) arc (180:360:2 and 0.65) node[anchor=west] {$$};
\draw[dashed] (2,0) arc (0:180:2 and 0.65);
\filldraw[blue] (0.8,-0.6) circle (0.8pt) node[anchor=north] {$$};
\filldraw[black] (0,2) circle (0.8pt) node[anchor=north] {$$};
\draw[blue, thick] (0,2).. controls (0.57, 1.5) and (0.798, 0.4) .. (0.8,-0.6);

\draw[red, thick] (0,2).. controls (1.95, 1.77) and (1.93, 0.4) .. (2,0);

\filldraw[black] (0,0) circle (0.8pt) node[anchor=south] {$O$};
\filldraw[red] (2,0) circle (0.8pt) node[anchor=south] {$$};
\draw[dashed] (0,0)--(0.8,-0.6);
\draw[blue ,thick, -] (0.8,-0.6)--(2.4, -1.8);
\draw[dashed] (0,0)--(2,0);
\draw[red,thick, -] (2,0)--(5,0);
\draw[green, thick] (2.4, -1.8).. controls (4.2, -1.6) and (5, -0.4) .. (5,0);
\filldraw[blue] (0.7,-1.2) circle (0pt) node[anchor=south] {$\gamma_{0,R}$};
\filldraw[red] (2.2,0.5) circle (0pt) node[anchor=south] {$\gamma_{\theta,R}$};
\filldraw[green] (4.6,-1.6) circle (0pt) node[anchor=south] {$C_{\theta, R}$};
\draw[gray, thin] (1.15,-0.6)--(1,-0.7);
\draw[gray, thin] (1.2,-0.75)--(1.45,-0.55);
\draw[gray, thin] (1.35,-0.85)--(2.25,-0.1);
\draw[gray, thin] (1.53,-0.95)--(2.5,-0.1);
\draw[gray, thin] (1.7,-1.08)--(2.8,-0.1);
\draw[gray, thin] (1.88,-1.22)--(3.15,-0.1);
\draw[gray, thin] (2.1,-1.34)--(3.52,-0.1);
\draw[gray, thin] (2.3,-1.53)--(4,-0.1);
\draw[gray, thin] (2.75,-1.55)--(4.5,-0.1);
\draw[gray, thin] (3.2,-1.53)--(4.9,-0.1);
\end{tikzpicture}

\caption{The contour $\Gamma_{\theta, R}$.}
\end{figure}
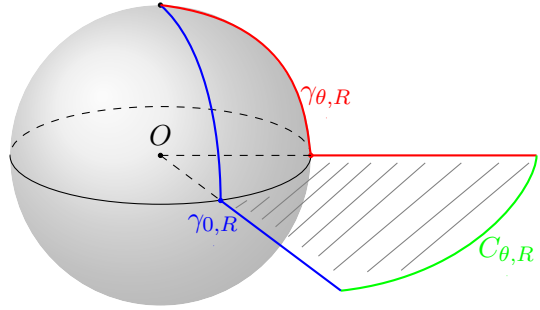

We note that the classical existence results on the Plateau problem guarantee the existence of a smooth area-minimizing disk $\Sigma_{\theta,R}$ with a prescribed boundary $\Gamma_{\theta, R}$. Additionally, we know that $\Sigma_{\theta, R}$ is embedded (see \cite{MY} for details) and has planar symmetry due to the symmetry of the contour across the plane $Q_{\frac{\theta}{2}}$. As $R$ tends to $\infty$, this Plateau solution serves as a fundamental piece that eventually becomes a complete embedded minimal surface after suitable patching using isometries. We record the following characteristics of $\Sigma_{\theta, R}$ for later use.

\begin{lemma}\label{convex}
Regardless of $R>\frac{m}{2}$, $\Sigma^{\text{int}}_{\theta,R}$ lies within the region enclosed by the three planes $Q_0$, $Q_\theta$, $P_0$, and the horizon $\partial M$.
\end{lemma}
\begin{proof}
Assume that a part of the surface $\Sigma_{\theta,R}$ lies outside of the region enclosed by $Q_0$, $Q_\theta$, $P_0$, and $\partial M$. We can obtain a contradiction by reflecting the outer part of the surface with respect to one of the totally geodesic surfaces. Smoothing the resulting surface leads to a contradiction with the fact that $\Sigma^{\text{int}}_{\theta,R}$ is an area-minimizing disk. Similarly, if $\Sigma^{\text{int}}_{\theta,R}$ touches the boundary of the region, we again get a contradiction using the maximum principle.
\end{proof}
\begin{rmk}\normalfont
This method is analogous to establishing the convex hull property for minimal surfaces in Euclidean spaces.
\end{rmk}

\begin{lemma}\label{monotone}
The family of area-minimizing disks $\{\Sigma_{\theta,R}\}$ has the following properties:
\begin{enumerate}
\item The family $\{\Sigma_{\theta,R}\}$ grows monotonically with respect to $R$ in the sense that $\Sigma_{\theta,{R_2}}^{\textup int}$ lies above $\Sigma_{\theta,{R_1}}^{\textup int}$ if $R_2> R_1>\frac{m}{2}$.
\item Each $\Sigma_{\theta,R}$ meets the totally geodesic plane $P_{\theta,\phi}$ along a simple curve that connects two points $A_{\theta,\phi}$ and $B_{\theta,\phi}$. Here, $A_{\theta,\phi}$ and $B_{\theta,\phi}$ are the points of intersection with $P_{\theta,\phi}$ and $\gamma_{0,R}, \gamma_{\theta,R}$, respectively, provided that $\phi\in\left(0,\frac{\pi}{2}\right)$.
\end{enumerate}
\end{lemma}
\begin{proof}
Both properties can be established by employing the cut-and-paste principle (a general version of this principle can be found in \cite{HMW}). To show the first property, we consider two positive real numbers $R_1$ and $R_2$ such that $\frac{m}{2}<R_1<R_2$. We want to show that the two surfaces $\Sigma_{\theta,{R_1}}$ and $\Sigma_{\theta,{R_2}}$ only intersect along the boundary curves $\gamma_{0,R_1}$ and $\gamma_{\theta,R_1}$ of $\Sigma_{\theta,{R_1}}$. Suppose the surfaces meet at an interior point $q$ of $\Sigma_{\theta,{R_1}}$. Regardless of whether they intersect tangentially or trasversely at $q$, we can find a piecewise smooth curve of intersection passing through $q$. However, note that these intersection curves cannot be connected to $C^{\text{int}}_{\theta, R_1}$ since $C^{\text{int}}_{\theta, R_1}\cap\Sigma_{\theta, R_2}=\emptyset$ due to Lemma \ref{convex}. This curve, possibly combined with other intersection curves or $\gamma_{0,R_1}\cup\gamma_{\theta,R_1}$, forms a piecewise smooth closed curve that bounds simply-connected regions in both surfaces. This immediately leads to a contradiction based on the cut-and-paste principle since both surfaces are area-minimizing disks.

For the second part of the lemma, the proof proceeds in the same manner. We observe that the totally geodesic planes $P_{\theta,\phi}$ intersect $\gamma_{0,R}$ and $\gamma_{\theta,R}$ at $A_{\theta,\phi}$ and $B_{\theta,\phi}$, respectively. For every point $p\in\Sigma^{\text{int}}_{\theta, R}\cap P_{\theta,\phi}$, it cannot be an endpoint of the intersection curves. If $\Sigma^{\text{int}}_{\theta, R}\cap P_{\theta,\phi}$ contains a piecewise smooth simple closed curve, we arrive at a contradiction as in the previous argument. Otherwise, every intersection curve must be connected to $A_{\theta,\phi}$ or $B_{\theta,\phi}$, which are the only points in $\partial\Sigma_{\theta,R}\cap P_{\theta,\phi}$. If $\Sigma_{\theta,R}\cap P_{\theta,\phi}$ were not a simple curve joining $A_{\theta,\phi}$ and $B_{\theta,\phi}$, there would again exist a simple closed curve in the intersection, leading to a contradiction.
\end{proof}
According to the above Lemma \ref{monotone}, we obtain an appropriate area bound for the surface $\Sigma_{\theta,R}$. To get the area bound, we introduce the domain $\triangle_\theta(\phi,\epsilon,\delta)$ as follows: Consider a domain in $M$ bounded by the totally geodesic planes $Q_0$, $Q_\theta$, $P_{\theta,\phi}$, and two concentric spheres whose Euclidean distances from the horizon are given by $\epsilon$ and $\delta$, respectively (see Figure 2). Here, $\phi\in\left(0,\frac{\pi}{2}\right)$ and $\frac{m}{2}<\epsilon<\delta$.
\begin{prop}\label{areaest}
For every $\Sigma_{\theta,R}$ ($R>\frac{m}{2}$) that intersects $\partial\triangle_\theta(\phi,\epsilon,\delta)$ transversally, there exists $K >0$ depending only on $\triangle_{\theta}(\phi,\epsilon,\delta)$ such that
\begin{align*}
\textup{Area}_g (\Sigma_{\theta,R} \cap \overline{\triangle_{\theta}(\phi,\epsilon,\delta)}) \le K.
\end{align*}
\end{prop}
\begin{proof}
Assume that $\Sigma_{\theta,R}$ intersects transversally with $\partial\triangle_\theta(\phi,\epsilon,\delta)$. We only need to consider the case where the domain satisfies $\triangle_{\theta}(\phi,\epsilon,\delta)^{\textup{int}}\cap \Sigma_{\theta,R} \neq \emptyset$. Note that the intersection curves $\partial \triangle_{\theta}(\phi,\epsilon,\delta) \cap \Sigma_{\theta,R}$ lie on $\Sigma_{\theta,R}^{\textup{int}}$ since $\partial \Sigma_{\theta,R} \cap \overline{\triangle_{\theta}(\phi,\epsilon,\delta)} = \emptyset$. Since $\Sigma_{\theta,R}$ is a disk-type Plateau solution, there exists a minimal immersion $\psi : D \rightarrow \Sigma^{\text{int}}_{\theta,R}$ where $D$ is an open unit disk. By the embeddedness of the surface and transversality, we see that $\psi^{-1}\left(\partial \triangle_{\theta}(\phi,\epsilon,\delta) \cap \Sigma_{\theta,R}\right)$ is a union of at most countably many disjoint closed curves in $D$. Thus, we can write
\begin{align*}
\psi^{-1}\left(\Sigma_{\theta, R}\cap\triangle_\theta(\phi,\epsilon,\delta)\right)=\cup_{i=1}^{\infty}D_{C_i},
\end{align*}
where $C_i$'s are the outermost curves in $D$, and $D_{C_i}$ denotes the closed region bounded by $C_i$. Here, ``outermost'' means that there is no other closed curve in $\psi^{-1}\left(\partial \triangle_{\theta}(\phi,\epsilon,\delta) \cap \Sigma_{\theta,R}\right)$ that contains $C$ in its interior.

We claim that for each $C_i$, one can associate a region $\Omega_{C_i}\subset\partial\triangle_\theta(\phi,\epsilon,\delta)$ bounded by $\psi(C_i)$ such that $\Omega_{C_i}\cap\Omega_{C_j}=\emptyset$ for all $i\neq j$. Since $C_1$ and $C_2$ are disjoint, we can define $\Omega_{C_1}$ and $\Omega_{C_2}$. Now assume that $\left\{\Omega_{C_1},\cdots,\Omega_{C_n}\right\}$ ($n\geq2$) is given. If $\Omega_{C_{n+1}}$ cannot be defined, then there should be at least one $\Omega_{C_*}$ in each component of $\partial\triangle_\theta(\phi,\epsilon,\delta)\setminus\psi(C_{n+1})$. Now there exist $C_i$ ($1\leq i\leq n$) and an angle $\phi_0 \in (\phi, \frac{\pi} {2})$ such that $P_{\theta,\phi_0} \cap \Sigma_{\theta,R}$ contains two curves in $\psi(D_{C_i})$ and $\psi(D_{C_{n+1}})$. By Lemma \ref{monotone} (2), we know that $\Sigma_{\theta,R}$ and $P_{\theta,\phi_0}$ meet along a simple curve joining two points $A_{\theta,\phi_0}$ and $B_{\theta,\phi_0}$. It means that this simple curve connects two regions $\psi(D_{C_i})$ and $\psi(D_{C_{n+1}})$ in $\Sigma_{\theta, R}$. This is impossible since $C_i$ and $C_{n+1}$ are two disjoint curves in the preimage $\psi^{-1}\left(\Sigma_{\theta, R}\cap\partial\triangle_\theta(\phi,\epsilon,\delta)\right)$.
Therefore the claim is proved.

Recall that the surface $\Sigma_{\theta,R}$ is an area-minimizing disk. Consequently, we can replace the area of $\psi(D_{C_i})$ with a larger area $\Omega_{C_i}$ to get the desired area bound
\begin{align*}
\textup{Area}_g (\Sigma_{\theta,R} \cap \overline{\triangle_{\theta}(\phi,\epsilon,\delta)}) \le K
\end{align*}
where $K$ is just the area of $\partial\triangle_\theta(\phi,\epsilon,\delta)$. This completes the proof.
\end{proof}
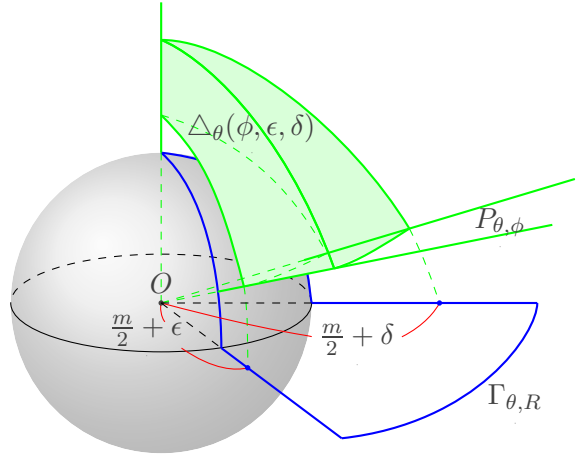
\begin{figure}
\centering
\begin{tikzpicture}
\shade[ball color = gray!40, opacity = 0.4] (0,0) circle (2cm);
\draw (-2,0) arc (180:360:2 and 0.65) node[anchor=west] {$$};
\draw[dashed] (2,0) arc (0:180:2 and 0.65);

\draw[blue, thick] (0,2).. controls (0.57, 1.5) and (0.798, 0.4) .. (0.8,-0.6);

\draw[blue, thick] (0,2).. controls (1.95, 1.77) and (1.93, 0.4) .. (2,0);

\fill [green!15, draw opacity=0.2] (0,2.5) -- (0,3.5) -- (0,3.5).. controls (1,3.5) and (2.5,2.3) .. (3.3,0.99) -- (3.3,0.99) .. controls (2.7,0.6) and (2.3,0.46) .. (2,0.4) -- (2,0.4) -- (1.1,0.22) -- (1.1,0.22).. controls (0.75,1.6)and(0.5,2) .. (0,2.5);

\draw[green,dashed, -] (0,0)--(0.75,0.15);
\draw[green,thick, -] (0.75,0.15)--(5.2,1.04);
\draw[green,dashed, -] (0,0)--(1.9,0.57);
\draw[green,thick, -] (1.9,0.57)--(5.5,1.65);

\draw[green,dashed, -] (0,0)--(0,2);
\draw[green,thick, -] (0,2)--(0,4);
\draw[green,thick, -] (0,2.5).. controls (0.5,2) and (0.75,1.6) .. (1.1,0.22);
\draw[green,dashed, -] (1.1,0.22).. controls (1.16,-0.1) and (1.15,-0.7) .. (1.15,-0.8625);

\draw[green,thick, -] (0,3.5).. controls (0.5,3.3) and (1.6,2.4) .. (2.3,0.46);

\draw[green,thick, -] (0,3.5).. controls (1,3.5) and (2.5,2.3) .. (3.3,0.99);
\draw[green,thick, -] (2.3,0.46).. controls (2.5,0.52) and (2.7,0.6) .. (3.3,0.99);
\draw[green,dashed, -] (0,2.5).. controls (1,2.2) and (1.7,1.6) .. (2.2,0.66);
\draw[green,dashed, -] (1.1,0.22).. controls (1.3,0.24) and (2,0.55) .. (2.2,0.66);
\draw[green,dashed, -] (3.3,0.99).. controls (3.5,0.6) and (3.6,0.3) .. (3.7,0);

\draw[red,thin,-] (0,0).. controls (-0.02,-0.05) and (0, -0.2).. (0.05,-0.25) ;
\draw[red,thin,-] (0.27,-0.6).. controls (0.7,-0.9) and (1,-0.9) .. (1.15,-0.8625);

\draw[red, thin,-] (0,0).. controls (1.5,-0.4) and (2,-0.43) .. (2,-0.45);
\draw[red, thin,-] (3.2,-0.45).. controls (3.5,-0.3) and (3.65,-0.1) .. (3.7,0);
\draw[green,thick, -] (0,3.5).. controls (0.5,3.3) and (1.6,2.4) .. (2.3,0.46);

\filldraw[darkgray] (-0.2,- 0.7) circle (0pt) node[anchor=south] {$\frac{m}{2}+\epsilon$};
\filldraw[darkgray] (2.6,- 0.82) circle (0pt) node[anchor=south] {$\frac{m}{2}+\delta$};

\filldraw[blue] (1.15,-0.8625) circle (0.8pt) node[anchor=south] {$$};
\filldraw[blue] (3.7,0) circle (0.8pt) node[anchor=south] {$$};
\filldraw[darkgray] (0,0) circle (0.8pt) node[anchor=south] {$O$};

\draw[dashed] (0,0)--(0.8,-0.6);
\draw[blue ,thick, -] (0.8,-0.6)--(2.4, -1.8);
\draw[dashed] (0,0)--(2,0);
\draw[blue,thick, -] (2,0)--(5,0);
\draw[blue,thick] (2.4, -1.8).. controls (4.2, -1.6) and (5, -0.4) .. (5,0);

\filldraw[darkgray] (1.2,2) circle (0pt) node[anchor=south] {$\triangle_{\theta}(\phi,\epsilon,\delta)$};
\filldraw[darkgray] (4.5,0.75) circle (0pt) node[anchor=south] {$P_{\theta,\phi}$};
\filldraw[darkgray] (4.7,-1.6) circle (0pt) node[anchor=south] {$\Gamma_{\theta,R}$};
\end{tikzpicture}

\caption{The domain $\triangle_{\theta}(\phi,\epsilon,\delta)$.}
\end{figure}
\begin{rmk}\normalfont
It is unsure whether the surface $\Sigma_{\theta,R}^{\text{\normalfont int}}$ would be a graph in the spherical coordinate system. If it is indeed a graph $\phi = \phi(r,\alpha) $ over an open domain, where $(r,\alpha, \phi)$ denotes the spherical coordinate system, the limit process in the following section can be simplified.
\end{rmk}

\section{A curvature estimate}\label{sec4}
\setcounter{equation}{0}
We establish the curvature estimate for the area-minimizing disk $\Sigma_{\theta,R}$ discussed in Section \ref{sec3}. From now on, $A^M_{\Sigma}$ and $A_{\Sigma}$ denote the second fundamental form of $\Sigma$ with respect to the metric $g$ and the Euclidean metric, respectively. Additionally, $|\cdot|_g$ refers to the norm relative to $g$.
\begin{prop}\label{curvest}
There exists $C>0$ such that
\begin{align*}
\sup_{R>\frac{m}{2}}\left(\sup_{q\in\Sigma_{\theta,R}}\left(\left|A^M_{\Sigma_{\theta,R}}(q)\right|_{g}\cdot\min\left(1,d_{\mathbb{R}^3}(q,\partial\Sigma_{\theta,R})\right)\right)\right)< C.
\end{align*}
\end{prop}
\begin{proof}
Suppose that $C>0$ does not exist. One can find a sequence $R_i\nearrow\infty$ such that
\begin{align}\label{supgoinf}
\sup_{q\in\Sigma_{\theta,{R_i}}}\left(\left|A^M_{\Sigma_{\theta,{R_i}}}(q)\right|_{g}\cdot\min\left(1,d_{\mathbb{R}^3}(q,\partial\Sigma_{\theta,{R_i}})\right)\right)\to\infty
\end{align}
as $i\to\infty$.
Since $\overline{\Sigma_{\theta,R_i}}$ is compact and $d_{\mathbb{R}^3}(q, \partial\Sigma_{\theta,{R_i}})=0$ for each $q\in\partial\Sigma_{\theta,{R_i}}$, there exists an interior point $q_i\in\Sigma_{\theta,{R_i}}$ where the supremum is attained. Let $\lambda_i:=|A^M_{\Sigma_{\theta,{R_i}}}(q_i)|_g$. Consequently, (\ref{supgoinf}) implies that $\lambda_i\to\infty$ as $i\to\infty$.

From now on, we consider $\Sigma_{\theta,{R_i}}$ as a surface in $\mathbb{R}^3$ with the Euclidean metric and define
\begin{align*}
\widetilde{\Sigma}_i:=\lambda_i\left(\Sigma_{\theta,{R_i}}-q_i\right).
\end{align*}
The mean curvature vector at $X\in\widetilde{\Sigma}_i$ is given by
\begin{align*}
\vec{H}_{\widetilde{\Sigma}_i}(X)&=\vec{H}_{\lambda_i\left(\Sigma_{\theta,{R_i}}-q_i\right)}(X)\\
&=\frac{1}{\lambda_i}\vec{H}_{\Sigma_{\theta,R_i}}\left(\frac{1}{\lambda_i}X+q_i\right)\\
&=\frac{-2m\lambda_i}{|X+\lambda_iq_i|^3\cdot\left(1+\frac{m\lambda_i}{2|X+\lambda_iq_i|}\right)}\cdot(X+\lambda_iq_i)^{\perp}.
\end{align*}
Since $\frac{1}{\lambda_i}X+q_i\in\Sigma_{\theta,R_i}\subset\mathbb{R}^3\backslash B_{\frac{m}{2}}(O)$, we have
\begin{align}\label{ddist}
\left|\frac{1}{\lambda_i}X+q_i\right|\geq\frac{m}{2},
\end{align}
which leads to
\begin{align*}
\left|\vec{H}_{\widetilde{\Sigma}_i}(X)\right|\leq\frac{2m\lambda_i\left|(X+\lambda_iq_i)^{\perp}\right|}{|X+\lambda_iq_i|^3\cdot\left(1+\frac{m\lambda_i}{2|X+\lambda_iq_i|}\right)}\leq\frac{4}{m\lambda_i}.
\end{align*}
Hence,
\begin{align}\label{Hgozero}
\sup_{\widetilde{\Sigma}_i}|\vec{H}_{\widetilde{\Sigma}_i}|\to0
\end{align}
as $i\to\infty$. On the other hand, one may compute
\begin{align}\label{AA}
|A_{\widetilde{\Sigma}_i}(X)|&=\left(1+\frac{m\lambda_i}{2|X+\lambda_iq_i|}\right)^2\cdot\frac{1}{\lambda_i}\cdot\left|A^M_{\Sigma_{\theta,R_i}}\left(\frac{1}{\lambda_i}X+q_i\right)\right|_g\\
&\leq 4\cdot\frac{1}{\lambda_i}\cdot\left|A^M_{\Sigma_{\theta,R_i}}\left(\frac{1}{\lambda_i}X+q_i\right)\right|_g,\nonumber
\end{align}
where the inequality follows from (\ref{ddist}). Fix a positive real number $l>0$. By (\ref{supgoinf}),
\begin{align*}
\lambda_id_{\mathbb{R}^3}(q_i,\partial\Sigma_{\theta, R_i})=d_{\mathbb{R}^3}(O,\partial\widetilde{\Sigma}_i)\to\infty,
\end{align*}
which implies that $\partial\widetilde{\Sigma}_i\cap B_l(O)=\emptyset$ for sufficiently large $i$. Therefore if $X\in\widetilde{\Sigma}_i\cap B_l(O)$, the assumption on $q_i$ yields
\begin{align*}
4\cdot\frac{1}{\lambda_i}\cdot\left|A^M_{\Sigma_{\theta,R_i}}\left(\frac{1}{\lambda_i}X+q_i\right)\right|_g&\leq4\cdot\frac{\min(1,d_{\mathbb{R}^3}(q_i,\partial\Sigma_{\theta,R_i}))}{\min(1,d_{\mathbb{R}^3}(\frac{1}{\lambda_i}X+q_i,\partial\Sigma_{\theta,R_i}))}\\
&=4\cdot\frac{\min(\lambda_i,d_{\mathbb{R}^3}(O,\partial\widetilde{\Sigma}_i))}{\min(\lambda_i,d_{\mathbb{R}^3}(X,\partial\widetilde{\Sigma}_i))},
\end{align*}
and by applying the triangle inequality, we obtain
\begin{align*}
4\cdot\frac{\min(\lambda_i,d_{\mathbb{R}^3}(O,\partial\widetilde{\Sigma}_i))}{\min(\lambda_i,d_{\mathbb{R}^3}(X,\partial\widetilde{\Sigma}_i))}\leq4\cdot\frac{d_{\mathbb{R}^3}(O, \partial\widetilde{\Sigma}_i)}{d_{\mathbb{R}^3}(O, \partial\widetilde{\Sigma}_i)-l}.
\end{align*}
Combining all the above inequalities, we get
\begin{align*}
|A_{\widetilde{\Sigma}_i}(X)|\leq4\cdot\frac{d_{\mathbb{R}^3}(O, \partial\widetilde{\Sigma}_i)}{d_{\mathbb{R}^3}(O, \partial\widetilde{\Sigma}_i)-l}
\end{align*}
for every $X\in\widetilde{\Sigma}_i\cap B_l(O)$, and thus we deduce that
\begin{align}\label{Aunifbdd}
\sup_{\widetilde{\Sigma}_i\cap B_l(O)}|A_{\widetilde{\Sigma}_i}|<8
\end{align}
for sufficiently large $i$.

Now the smooth compactness theorem can be applied to $\widetilde{\Sigma}_i\cap B_l(O)$'s by (\ref{Hgozero}) and (\ref{Aunifbdd}). Furthermore, by applying diagonal arguments, we can eventually obtain a subsequence $\widetilde{\Sigma}_i$ (using the same notation) that converges to a smooth complete minimal surface $\widetilde{\Sigma}_{\infty}$ in $\mathbb{R}^3$. By Lemma \ref{stabinf} below, $\widetilde{\Sigma}_{\infty}$ is also stable. Then the Bernstein theorem implies that $\widetilde{\Sigma}_{\infty}$ must be a plane. However, in (\ref{AA}),
\begin{align*}
|A_{\widetilde{\Sigma}_i}(O)|=\left(1+\frac{m}{2|q_i|}\right)^2,
\end{align*}
and therefore $|A_{\widetilde{\Sigma}_{\infty}}(O)|\geq 1$. This is a contradiction.
\end{proof}

\begin{lemma}\label{stabinf}
$\widetilde{\Sigma}_{\infty}$ is a stable minimal surface in $\mathbb{R}^3$.
\end{lemma}
\begin{proof}
Consider the following stability inequality of $\Sigma_{\theta,R}$ in $M$:
\begin{align*}
\int_{\Sigma_{\theta,R}}\left(|A^M_{\Sigma_{\theta,R}}|_g^2+\text{Ric}_g(\nu_g,\nu_g)\right)f^2\leq\int_{\Sigma_{\theta,R}}|\nabla^{\Sigma_{\theta,R}}f|_g^2
\end{align*}
for every $f\in C^{\infty}_c(\Sigma_{\theta,R})$. Here, $\nu_g$ denotes the unit normal vector of $\Sigma_{\theta,R}$ in $M$. Since the Ricci curvature of $M$ is the smallest in the radial direction, we have
\begin{align}\label{ricci}
-\text{Ric}_g(\nu_g,\nu_g)(X)\leq\frac{2m}{|X|^3\left(1+\frac{m}{2|X|}\right)^6}
\end{align}
at each $X\in\Sigma_{\theta,R}$, where the right-hand side is the negative of the Ricci curvature along the radial direction. Substituting (\ref{ricci}) into the above stability inequality and expressing all the terms with respect to the Euclidean metric, it follows that
\begin{align*}
\int_{\Sigma_{\theta,R}}\left(|A_{\Sigma_{\theta,R}}|^2-\frac{2m^2(X\cdot\nu)^2}{|X|^6\left(1+\frac{m}{2|X|}\right)^2}\right)f^2\leq \int_{\Sigma_{\theta,R}}|\nabla^{\Sigma_{\theta,R}}f|^2+\frac{2m}{|X|^3\left(1+\frac{m}{2|X|}\right)^2}f^2,
\end{align*}
where $\nu$ denotes the unit normal vector in the Euclidean space. Let $\tilde{f}(X):=f\left(\frac{1}{\lambda_i}X+q_i\right)$ for $X\in\widetilde{\Sigma}_i$. Then $\tilde{f}\in C^{\infty}_c(\widetilde{\Sigma}_i)$, and the inequality eventually becomes
\begin{align*}
\int_{\widetilde{\Sigma}_i}\left(|A_{\widetilde{\Sigma}_i}|^2-\frac{2m^2(X\cdot\nu+\lambda_i(q_i\cdot\nu))^2}{|X+\lambda_iq_i|^6\left(\frac{1}{\lambda_i}+\frac{m}{2|X+\lambda_iq_i|}\right)^2}\right)\tilde{f}^2\\\leq\int_{\widetilde{\Sigma}_i}|\nabla^{\widetilde{\Sigma}_i}\tilde{f}|^2+\frac{2m}{\lambda_i|X+\lambda_iq_i|^3\left(\frac{1}{\lambda_i}+\frac{m}{2|X+\lambda_iq_i|}\right)^2}\tilde{f}^2.
\end{align*}
Since $\widetilde{\Sigma}_i$ converges to $\widetilde{\Sigma}_{\infty}$ smoothly and $\lambda_i\to\infty$ as $i\to\infty$, we deduce that
\begin{align*}
\int_{\widetilde{\Sigma}_{\infty}}|A_{\widetilde{\Sigma}_{\infty}}|^2f^2\leq\int_{\widetilde{\Sigma}_{\infty}}|\nabla^{\widetilde{\Sigma}_{\infty}}f|^2
\end{align*}
for every $f\in C_c^{\infty}(\widetilde{\Sigma}_{\infty})$. Therefore $\widetilde{\Sigma}_{\infty}$ is stable.
\end{proof}

\begin{rmk}
\normalfont The idea of performing scaling and the limiting process in the Euclidean space rather than in the Schwarzschild manifold was motivated by \cite{HMW}. What is new in our proof is that the stability inequality of $\Sigma_{\theta,R}$ is also examined in the Euclidean space, as shown in the proof of Lemma \ref{stabinf}.
\end{rmk}


\section{Construction of a fundamental piece}\label{sec5}
\setcounter{equation}{0}
In this section, we investigate the limit of $\Sigma_{\theta,R}$ as $R\to\infty$. We restrict our attention to angles $\theta$ of the form $\frac{\pi}{N}$, where $N\geq2$ is an integer. Let $\{\Sigma_{\theta,R_i}\}_{i=1}^{\infty}$ be an arbitrary sequence such that $R_i\nearrow\infty$. To use the curvature estimate from Proposition \ref{curvest}, we introduce an increasing sequence of domains $\triangle_{\theta,j}:=\triangle_\theta(\phi_j,\epsilon_j,\delta_j)$ (defined in Section \ref{sec3}) such that $\phi_j, \epsilon_j$, and $\delta_j$ are sequences of positive real numbers satisfying
\begin{align*}
\phi_j \searrow 0, \quad \epsilon_j \searrow 0 , \quad \delta_j \nearrow \infty, \quad \epsilon_j < \delta_j.
\end{align*}
We may assume that $\Sigma_{\theta,R_i}$ and $\partial\triangle_{\theta,j}$ intersect transversally for all $i,j\geq1$.

Let us define $W_{\theta}:=\cup_{j=1}^{\infty}\triangle_{\theta,j}$. For each $j$, the distances between $\triangle_{\theta,j}$ and $\partial\Sigma_{\theta,R_i}$ are bounded below by a positive constant depending only on $j$. Hence, Proposition \ref{curvest} implies the existence of $C_j>0$ such that
\begin{align}\label{curvestj}
\sup_{\Sigma_{\theta,R_i}\cap\triangle_{\theta,j}}\left|A^M_{\Sigma_{\theta,R_i}}\right|_{g}<C_j
\end{align}
for all $i\geq1$.
Combining the curvature estimate (\ref{curvestj}) with the area estimate in Proposition \ref{areaest}, we can apply the smooth compactness theorem in each $\triangle_{\theta,j}$. By employing a diagonal argument, we obtain a subsequence $\Sigma_{\theta,R_i}$ (using the same notation) such that $\Sigma_{\theta, R_i}$ converges smoothly to $\Sigma_\theta$ in every compact subset of $W_\theta$ as $i\to\infty$.

We observe that $\Sigma^{\text{int}}_{\theta}$ has a finite number of connected components. Indeed, let $\widetilde{\triangle}_{\theta,j}$ be a domain defined similarly to $\triangle_{\theta,j}$ but with constants $R_1$ and $R_2$ instead of $\epsilon_j$ and $\delta_j$. It suffices to show that $\Sigma_\theta\cap K$ has finitely many components, where $K:=\left(\cup_{j=1}^{\infty}\overline{\widetilde{\triangle}_{\theta,j}}\right)$. Proposition \ref{areaest} implies that $\Sigma_\theta\cap\overline{\widetilde{\triangle}_{\theta,j}}$ has a finite number of components. If there were infinitely many components in $\Sigma_\theta\cap K$, there should be new components in $K\setminus\widetilde{\triangle}_{\theta,j}$ for infinitely many $j$. This implies that these new components approach $P_0$ as $j$ increases. However, by Lemma \ref{monotone}, for some $R_0$, $\Sigma_{\theta,R_0}\cap K$ acts as a barrier separating $P_0$ from $\Sigma_\theta\cap K$. This is a contradiction. Therefore, we have the following convergence result:
\begin{align}\label{multlimit}
\Sigma_{\theta,R_i}\to\Sigma_\theta:=n_1T_1+\cdots+n_lT_l
\end{align}
smoothly in every compact subset of $W_\theta$ as $i\to\infty$. Here, $T_i$'s are connected smooth oriented minimal surfaces.

We note that a similar convergence result can be obtained using varifold convergence. In \cite{AS}, interior regularity for the limit of an almost minimizing sequence of disks was established using the filigree lemma and the replacement theorem (Lemma 3 and Theorem 1 in \cite{AS}). In \cite{MSY}, it was demonstrated how the arguments of \cite{AS} can be modified to homogeneous regular $3$-manifolds, proving the persistence of interior regularity. It is evident that the Schwarzschild $3$-manifold is homogenous regular. Nevertheless, to obtain the complete description of the limit surface as in (\ref{multlimit}), it is necessary to employ similar arguments throughout this paper.

However, to study the behavior of $\Sigma_\theta$ at the boundary $\Gamma_\theta:=\bigcup_{R>\frac{m}{2}}\gamma_{0,R}\cup\gamma_{\theta,R}$, we need to consider it as a varifold. We can obtain boundary regularity similar to that demonstrated in \cite{AS} with the same modifications as in \cite{MSY}. Let $M$ be isometrically embedded in some $\mathbb{R}^{N_0}$, and let $p\in\Gamma_\theta$. As in \cite{AS}, denote
\begin{align*}
C_+=\lim_{k\to\infty}\mu_{r_k\#}v(\Sigma_\theta)\ (r_k\to\infty\ \text{as}\ k\to\infty),
\end{align*}
where $v(\Sigma_\theta)$ is a varifold corresponding to $\Sigma_\theta$, $\mu_r:\mathbb{R}^{N_0}\to\mathbb{R}^{N_0}$ is a map defined by $\mu_r(x)=r(x-p)$, and $C_+$ is any varifold tangent to $v(\Sigma_\theta)$ at $p\in\Gamma_\theta$. The main difference between our case and \cite{AS} is that we have to send it to the limit while applying the filigree lemma and the replacement theorem in each $\mu_{r_k}(M)$. Since $M$ is simply enlarged, Lemma 1 and Lemma 3 in \cite{MSY} can be uniformly applied to $\mu_{r_k}(M)$ for all sufficiently large $k$. Furthermore, $\Gamma_\theta$ is a piecewise smooth geodesic in $M$ with angles of the form $\frac{\pi}{l}$ ($l\in\mathbb{Z}_{>0}$) at singular points, and it lies on the boundary of the convex domain $W_\theta$. Therefore, a similar argument as in \cite{AS} can be applied to prove that $C_{+}=\sum_{i=1}^{s}H_i$, where $H_i\subset T_pM\subset\mathbb{R}^{N_0}$ is a convex cone in some $2$-plane such that $\partial H_i$ is a tangent cone of $\Gamma_\theta$ at $p$. Moreover, calculating the density of $C_+$ at $p$ in a similar manner as in \cite{AS} gives $s\leq2$, which implies that $s=1$ due to orientation. Thus, in (\ref{multlimit}), we have the following result:
\begin{lemma}
$\Sigma_\theta=T_1$.
\end{lemma}
Furthermore, we can show that $T_1$ is simply-connected.
\begin{lemma}
$T_1$ is simply-connected.
\end{lemma}
\begin{proof}
Suppose the contrary. We may assume that there is a Jordan curve $\gamma$ in $T_1$ that represents a non-trivial homotopy class. Then, for sufficiently large $i$, there exists a simple closed curve $\gamma_i$ in each $\Sigma_{\theta,R_i}$ that converges to $\gamma$. Let $D_i\subset\Sigma_{\theta,R_i}$ be a region bounded by $\gamma_i$. Since $D_i$'s are simply-connected, for $\gamma_i$ to converge to a homotopically non-trivial curve $\gamma$ in $T_1$, $\cup_{i=1}^{\infty}D_i$ must be unbounded. This implies that one can find a region $B_r\subset M$ given by $\left\{x\in\mathbb{R}^3:\frac{m}{2}\leq|x|< r\right\}$ and $D_j\subset\Sigma_{\theta,R_j}$ such that $\gamma_j=\partial D_j\subset B_r$ and $D_j\cap\left(M\backslash\overline{B_r}\right)\neq\emptyset$. Now, if we decrease the radius $L$ of the sphere $\{x\in\mathbb{R}^3: |x|=L\}$ from infinity, we can consider the first touching point with $D_j$, which lies in the interior of $D_j$. Since the sphere has mean curvature vector towards $\partial M$, this contradicts the maximum principle.
\end{proof}
By the monotonicity property proved in Lemma \ref{monotone}, it always converges to the same surface $\Sigma_\theta$ regardless of the choice of a sequence $\Sigma_{\theta,R_i}$. As a consequence, we obtain the following theorem:
\begin{thm}\label{piece}
There exists a simply-connected smooth embedded non-compact minimal surface $\Sigma_\theta\subset W_\theta$ with $\partial \Sigma_\theta=\Gamma_\theta$. Furthermore, $\Sigma_\theta$ has finite total curvature and quadratic area growth.
\end{thm}
\begin{proof}
It only remains to show that $\Sigma_\theta$ has finite total curvature and quadratic area growth. Consider area-minimizing disks $\Sigma_{\theta,R_1}$ and $\Sigma_{\theta,R_2}$ defined in Section \ref{sec3}, where $\frac{m}{2}<R_1<R_2$.
The area-minimizing property implies
\begin{align*}
\text{Area}_g\left(\Sigma_{\theta,R_2}\right)-\text{Area}_g\left(A(R_1, R_2)\cap W_{\theta}\right)&<\text{Area}_g\left(\Sigma_{\theta,R_1}\right)\\
&<\text{Area}_g\left(\Sigma_{\theta,R_2}\right)+\text{Area}_g\left(A(R_1, R_2)\cap W_{\theta}\right),
\end{align*}
where $A(R_1,R_2)$ denotes the annular region in $P_0$ bounded by the spheres $\{x\in\mathbb{R}^3: |x|=R_1\}$ and $\{x\in\mathbb{R}^3: |x|=R_2\}$. As $M$ is asymptotically Euclidean, $\Sigma_\theta$ has quadratic area growth. On the other hand, the smooth convergence implies the stability of $\Sigma_\theta$. Recall that the stability inequality writes
\begin{align*}
\int_{\Sigma_\theta}\left(|A^M_{\Sigma_\theta}|_g^2+\text{Ric}_g(\nu_g,\nu_g)\right)f^2\leq\int_{\Sigma_\theta}|\nabla^{\Sigma_\theta}f|_g^2,
\end{align*}
and from (\ref{ricci}), it follows that
\begin{align*}
\int_{\Sigma_\theta}|A^M_{\Sigma_\theta}|_g^2f^2\leq\int_{\Sigma_\theta}|\nabla^{\Sigma_\theta}f|_g^2+\frac{2m}{|X|^3\left(1+\frac{m}{2|X|}\right)^6}f^2
\end{align*}
for every $f\in C_c^{\infty}(\Sigma_\theta)$. The right-most term is of the order $O\left(\frac{1}{|X|^3}\right)$. Therefore, the quadratic area growth of $\Sigma_\theta$ and the logarithmic cut-off trick imply that $\Sigma_\theta$ has finite total curvature.
\end{proof}

\section{Proof of main theorem}\label{sec6}
\setcounter{equation}{0}
In this section, we construct complete embedded minimal surfaces in the doubled Schwarzschild $3$-manifold $\widehat{M}$. As $\partial\Sigma_\theta$ consists of geodesic lines, we use the reflection principle to patch up copies of $\Sigma_\theta$. This technique has a significant historical importance, and many embedded minimal surfaces have been constructed using this approach (see \cite{Choe, HM2, HMW, L} for example).

Let $\tau$ be a positive integer, and consider the contour $\Gamma_{\frac{\pi}{\tau+1}}$ defined in Section \ref{sec5}. By Theorem \ref{piece}, we have the simply-connected minimal surface $\Sigma_{\frac{\pi}{\tau+1}}$. Now we introduce three particular isometries of $\widehat{M}$ as follows:
\begin{align*}
a=\text{Ref}_{Q_0}\circ I,\ b_\tau=\text{Ref}_{Q_{\frac{\pi}{\tau+1}}}\circ\text{Ref}_{P_{\frac{\pi}{\tau+1},0}},\ c_\tau=\text{Rot}_{\frac{2\pi}{\tau+1}},
\end{align*}
where $\text{Ref}_*$ denotes the reflection across the totally geodesic plane $*$, $\text{Rot}_{\frac{2\pi}{\tau+1}}$ is the rotation of $\frac{2\pi}{\tau+1}$ around the vector $(0, 0, 1)$, and $I$ is the inversion defined in Section \ref{pre}. Then $a$, $b_\tau$, and $c_\tau$ generate a group $G_\tau$ of isometries of order $2\tau+2$. We define
\begin{align*}
\Sigma_\tau :=\bigcup_{h\in G_\tau}h\cdot\Sigma_{\frac{\pi}{\tau+1}}.
\end{align*}

\begin{thm}\label{mainthm}
For each integer $\tau\geq1$, there exists a complete embedded minimal surface $\Sigma_{\tau}\subset\widehat{M}$ of genus $\tau$ satisfying the following properties:
\begin{enumerate}
\item $\Sigma_{\tau}$ has finite total curvature and quadratic area growth;
\item $\Sigma_{\tau}$ is asymptotic to the totally geodesic plane $P_0$;
\item $\Sigma_{\tau}\cap \partial M=\cup_{j=1}^{2\tau+2}\eta_i$, where $\eta_i$'s are geodesics on $\partial M$ passing through two poles and making equal angles of $\theta=\frac{\pi}{\tau+1}$;
\item $\Sigma_{\tau}\cap P_0=\cup_{i=1}^{2\tau+2}l_{i}$, where $l_i's$ are geodesic lines on $P_0$ perpendicular to $\partial M$ and $\eta_i$ and making  equal angles of $\theta=\frac{\pi}{\tau+1}$ at infinity; and
\item $\Sigma_{\tau}$ has a dihedral group of symmetries of order $2\tau+2$ generated by rotations across $l_i$ and $\eta_j$. $\Sigma_r$ has genus $\tau=\frac{\pi}{\theta}-1$.
\end{enumerate}
\end{thm}
\begin{proof}
By the construction process, (3), (4), and (5) are clear. Note that $\eta_i$'s and $l_i$'s satisfy $\cup_{i=1}^{2\tau+2}\left(\eta_i\cup l_i\right)=\cup_{h\in G_\tau}h\cdot\Gamma_{\frac{\pi}{\tau+1}}$. Since $\Sigma_\tau$ is a finite union of isometric copies of $\Sigma_{\frac{\pi}{\tau+1}}$, (1) follows from Theorem \ref{piece}. On the other hand, in \cite[Corollary 1.2]{BR}, asymptotic behaviors of complete embedded minimal ends of finite total curvature and quadratic area growth in asymptotically flat spaces have been classified: the ends must either be bounded or have logarithmic growth. Properties (1) and (4) imply that $\Sigma_\tau$ has a planar end. Finally, the genus of $\Sigma_\tau$ can be easily calculated in terms of $\theta$ by the Gauss-Bonnet theorem.
\end{proof}


\end{document}